\documentclass[12pt]{amsart}
\usepackage{a4wide,hyperref}

\theoremstyle{plain}
\newtheorem{theorem}{Theorem}
\newtheorem{corollary}{Corollary}
\newtheorem{lemma}{Lemma}
\theoremstyle{remark}
\newtheorem{remark}{Remark}
\numberwithin{equation}{section}

\title{Digital expansions with negative real bases}

\author{Wolfgang Steiner}
\address{LIAFA, CNRS UMR 7089, Universit\'e Paris Diderot -- Paris 7,
Case 7014, 75205 Paris Cedex 13, FRANCE}
\email{steiner@liafa.jussieu.fr}
\thanks{Part of this research was conducted while the author was visiting academic at the Department of Computing of the Macquarie University, Sydney.}

\begin{document}
\begin{abstract}
Similarly to Parry's characterization of $\beta$-expansions of real numbers in real bases $\beta > 1$, Ito and Sadahiro characterized digital expansions in negative bases, by the expansions of the endpoints of the fundamental interval.
Parry also described the possible expansions of $1$ in base $\beta > 1$.
In the same vein, we characterize the sequences that occur as $(-\beta)$-expansion of $\frac{-\beta}{\beta+1}$ for some $\beta > 1$. 
These sequences also describe the itineraries of $1$ by linear mod one transformations with negative slope.
\end{abstract}
\maketitle

\section{Introduction}
Digital expansions in real bases $\beta > 1$ were introduced by R\'enyi~\cite{Renyi57}: 
The (greedy) \emph{$\beta$-expansion} of a real number $x \in [0,1)$ is 
\[
x = \frac{\varepsilon_1(x)}{\beta} + \frac{\varepsilon_2(x)}{\beta^2} + \cdots \quad \mbox{with} \quad  \varepsilon_n(x) = \lfloor \beta\, T_\beta^{n-1}(x) \rfloor,
\]
where $\lfloor \cdot \rfloor$ denotes the floor function and $T_\beta$ is the \emph{$\beta$-transformation}
\[
T_\beta:\, [0,1) \to [0,1),\quad x \mapsto \beta x - \lfloor \beta x \rfloor\,.
\]
R\'enyi suggested representing arbitrary $x \in \mathbb{R}$ by
\[
x = \lfloor x \rfloor + \frac{\varepsilon_1(\lfloor x\rfloor)}{\beta} + \frac{\varepsilon_2(\lfloor x\rfloor)}{\beta^2} + \cdots,
\]
whereas nowadays it is more usual (for $x \ge 0$) to multiply the $\beta$-expansion of $x \beta^{-k}$ by~$\beta^k$, with $k$ an arbitrary integer satisfying $x \beta^{-k} \in [0,1)$. 
Anyway, the possible expansions can be described by those of $x \in [0,1)$. 
A~sequence $b_1 b_2 \cdots$ is called \emph{$\beta$-admissible} if and only if it is (the digit sequence of) the $\beta$-expansion of a number $x \in [0,1)$, i.e., $b_n = \varepsilon_n(x)$ for all $n \ge 1$.
Parry~\cite{Parry60} showed that an integer sequence $b_1 b_2 \cdots$ is $\beta$-admissible if and only~if 
\[
0 0 \cdots \le_\mathrm{lex} b_k b_{k+1} \cdots <_\mathrm{lex} a_1 a_2 \cdots \quad \mbox{for all}\ k \ge 1,
\]
where $<_\mathrm{lex}$ denotes the lexicographic order and $a_1 a_2 \cdots$ is the (quasi-greedy) \emph{$\beta$-expansion of~$1$}, i.e., $a_n = \lim_{x\to1-} \varepsilon_n(x)$.
Moreover, a sequence of integers $a_1 a_2 \cdots$ is the (quasi-greedy) $\beta$-expansion of~$1$ for some $\beta > 1$ if and only if 
\[
0 0 \cdots <_\mathrm{lex} a_k a_{k+1} \cdots \le_\mathrm{lex} a_1 a_2 \cdots \quad \mbox{for all}\ k \ge 2.
\]
(These results are stated in a slightly different way in~\cite{Parry60}.)

Following \cite{Renyi57} and \cite{Parry60}, a lot of papers were dedicated to the study of $\beta$-expansions and $\beta$-transformations, but surprisingly little attention was given to digital expansions in negative bases. 
This changed only in recent years, after Ito and Sadahiro \cite{Ito-Sadahiro09} considered \emph{$(-\beta)$-expansions}, $\beta > 1$, defined for $x \in \big[\frac{-\beta}{\beta+1}, \frac{1}{\beta+1}\big)$ by
\begin{equation} \label{e:epsilon}
x = \frac{\varepsilon_1(x)}{-\beta} + \frac{\varepsilon_2(x)}{(-\beta)^2} + \cdots \quad \mbox{with} \quad  \varepsilon_n(x) = \big\lfloor \tfrac{\beta}{\beta+1} - \beta\, T_{-\beta}^{n-1}(x) \big\rfloor,
\end{equation}
where the \emph{$(-\beta)$-transformation} is defined by
\[
T_{-\beta}:\, \big[\tfrac{-\beta}{\beta+1}, \tfrac{1}{\beta+1}\big) \to \big[\tfrac{-\beta}{\beta+1}, \tfrac{1}{\beta+1}\big),\quad x \mapsto -\beta x - \big\lfloor \tfrac{\beta}{\beta+1} - \beta x \big\rfloor.
\]
A~sequence $b_1 b_2 \cdots$ is \emph{$(-\beta)$-admissible} if and only if it is the $(-\beta)$-expansion of some $x \in \big[\frac{-\beta}{\beta+1}, \frac{1}{\beta+1}\big)$, i.e., $b_n = \varepsilon_n(x)$ for all $n \ge 1$.
Since the map $x \mapsto -\beta x$ is order-reversing, the $(-\beta)$-admissible sequences are characterized using the alternating lexicographic order. 
By \cite{Ito-Sadahiro09}, a sequence $b_1 b_2 \cdots$ is $(-\beta)$-admissible if and only if
\begin{equation} \label{e:IS1}
a_1 a_2 \cdots \ge_\mathrm{alt} b_k b_{k+1} \cdots >_\mathrm{alt} 0 a_1 a_2 \cdots\quad \mbox{for all}\ k \ge 1,
\end{equation}
where $a_1 a_2 \cdots$ is the $(-\beta)$-expansion of the left endpoint~$\frac{-\beta}{\beta+1}$, i.e., $a_n = \varepsilon_n\big(\frac{-\beta}{\beta+1}\big)$, which is supposed not to be periodic with odd period length.
If $a_1 a_2 \cdots = \overline{a_1 a_2 \cdots a_{2\ell+1}}$ for some $\ell \ge 0$, and $\ell$ is minimal with this property, then the condition \eqref{e:IS1} is replaced by
\begin{equation} \label{e:IS2}
a_1 a_2 \cdots \ge_\mathrm{alt} b_k b_{k+1} \cdots >_\mathrm{alt} \overline{0 a_1 \cdots a_{2\ell} (a_{2\ell+1}\!-\!1)} \quad \mbox{for all}\ k \ge 1.
\end{equation}
Recall that the alternating lexicographic order is defined on sequences $x_1 x_2 \cdots$, $y_1 y_2 \cdots$ with $x_1 \cdots x_{k-1} = y_1 \cdots y_{k-1}$ and $x_k \ne y_k$ by 
\[
x_1 x_2 \cdots <_\mathrm{alt} y_1 y_2 \cdots \quad \mbox{if and only if} \quad \begin{cases}x_k < y_k & \mbox{when $k$ is odd}, \\[.5ex] y_k < x_k & \mbox{when $k$ is even}.\end{cases}
\]

The main result of this paper is a characterization of the sequences $a_1 a_2 \cdots$ that are the $(-\beta)$-expansion of~$\frac{-\beta}{\beta+1}$ for some $\beta > 1$.
This turns out to be more complicated than the corresponding problem for $\beta$-expansions, and we will see that several proofs cannot be directly carried over from positive to negative bases.
From~\eqref{e:IS1} and~\eqref{e:IS2}, one deduces~that 
\begin{equation} \label{e:gora}
a_k a_{k+1} \cdots \le_\mathrm{alt} a_1 a_2 \cdots \quad \mbox{for all}\ k \ge 2.
\end{equation}
The proof of Proposition~3.5 in \cite{Liao-Steiner} (see also Theorem~\ref{t:order} below) shows that 
\begin{equation} \label{e:LS}
a_1 a_2 \cdots >_\mathrm{alt} u_1 u_2 \cdots = 1 0 0 1 1 1 0 0 1 0 0 1 0 0 1 1 1 0 0 1 1 \cdots,
\end{equation}
where $u_1 u_2 \cdots$ is the sequence starting with $\varphi^n(1)$ for all $n \ge 0$, with $\varphi$ being the morphism of words on the alphabet $\{0,1\}$  defined by $\varphi(1) = 100$, $\varphi(0) = 1$.
(See the remarks following Theorem~\ref{t:order} and note that the alphabet is shifted by~$1$ in~\cite{Liao-Steiner}.)
Our first result states that a sequence satisfying \eqref{e:gora} and~\eqref{e:LS} is ``almost'' the $(-\beta)$-expansion of~$\frac{-\beta}{\beta+1}$ for some~$\beta > 1$.

\begin{theorem} \label{t:gora}
Let $a_1 a_2 \cdots$ be a sequence of non-negative integers satisfying \eqref{e:gora} and~\eqref{e:LS}. 
Then there exists a unique $\beta > 1$ such that 
\begin{equation} \label{e:closedinterval}
\sum_{j=1}^\infty \frac{a_j}{(-\beta)^j} = \frac{-\beta}{\beta+1} \quad \mbox{and} \quad \sum_{j=1}^\infty \frac{a_{k+j}}{(-\beta)^j} \in \bigg[\frac{-\beta}{\beta+1}, \frac{1}{\beta+1}\bigg] \quad \mbox{for all}\ k \ge 1. 
\end{equation}
\end{theorem}

For a $(-\beta)$-expansion of $\frac{-\beta}{\beta+1}$, we have to exclude the possibility that $\sum_{j=1}^\infty \frac{a_{k+j}}{(-\beta)^j} = \frac{1}{\beta+1}$ for some $k \ge 1$. 
If $\overline{a_1 \cdots a_k} >_\mathrm{alt} u_1 u_2 \cdots$, then out of $\{a_1 \cdots a_k,\, a_1 \cdots a_{k-1} ({a_k\!-\!1}) 0\}^\omega$, which is the set of infinite sequences composed of blocks $a_1 \cdots a_k$ and $a_1 \cdots a_{k-1} ({a_k\!-\!1}) 0$, only the periodic sequence $\overline{a_1 \cdots a_k}$ is possibly the $(-\beta)$-expansion of~$\frac{-\beta}{\beta+1}$ for some $\beta > 1$, see Section~\ref{sec:proof-theor-reft:ch}.
This implies that 
\begin{align} 
& a_1 a_2 \cdots \not\in \{a_1 \cdots a_k,\, a_1 \cdots a_{k-1} ({a_k\!-\!1}) 0\}^\omega \setminus \{\overline{a_1 \cdots a_k}\} \label{e:ak1} \\ 
& \hspace{12em} \mbox{for all}\ k \ge 1\ \mbox{with}\ \overline{a_1 \cdots a_k} \succ u_1 u_2 \cdots, \nonumber \\
& a_1 a_2 \cdots \not\in \{a_1 \cdots a_k 0,\, a_1 \cdots a_{k-1} ({a_k\!+\!1})\}^\omega \label{e:ak2} \\
& \hspace{12em} \mbox{for all}\ k \ge 1\ \mbox{with}\ \overline{a_1 \cdots a_{k-1} ({a_k\!+\!1})} \succ u_1 u_2 \cdots. \nonumber
\end{align}
The main result states that there are no other conditions on $a_1 a_2 \cdots$. 

\begin{theorem} \label{t:char}
A~sequence of non-negative integers $a_1 a_2 \cdots$ is the $(-\beta)$-expansion of~$\frac{-\beta}{\beta+1}$ for some (unique) $\beta > 1$ if and only if it satisfies \eqref{e:gora}, \eqref{e:LS}, \eqref{e:ak1}, and~\eqref{e:ak2}.
\end{theorem}

It is easy to see that the natural order of bases $\beta > 1$ is reflected by the lexicographical order of the (quasi-greedy) $\beta$-expansions of~$1$~\cite{Parry60}.
For negative bases, a similar relation with the alternating lexicographic order holds, although it is a bit harder to prove.

\begin{theorem} \label{t:order}
Let $a_1 a_2 \cdots$ be the $(-\beta)$-expansion of $\frac{-\beta}{\beta+1}$ and $a'_1 a'_2 \cdots$ be the $(-\beta')$-expansion of $\frac{-\beta'}{\beta'+1}$, with $\beta, \beta' > 1$.
Then $\beta < \beta'$ if and only if $a_1 a_2 \cdots <_\mathrm{alt} a'_1 a'_2 \cdots$.
\end{theorem}

It is often convenient to study a slightly different $(-\beta)$-transformation, 
\[
\widetilde{T}_{-\beta}:\, (0,1] \to (0,1], \quad x \mapsto -\beta x + \lfloor \beta x \rfloor + 1.
\]
As already noted in~\cite{Liao-Steiner}, the transformations $T_{-\beta}$ and $\widetilde{T}_{-\beta}$ are conjugate via the involution $\phi(x) = \frac{1}{\beta+1} - x$, i.e., 
\[
T_{-\beta} \circ \phi(x) = \phi \circ \widetilde{T}_{-\beta}(x) \quad \mbox{for all}\ x \in (0,1].
\]
Setting $\tilde\varepsilon_n(x) = \big\lfloor \beta\, \widetilde{T}_{-\beta}^{n-1}(x) \big\rfloor$ for $x \in (0,1]$, we have $x = - \sum_{n=1}^\infty \frac{\tilde\varepsilon_n(x)+1}{(-\beta)^n} = \frac{1}{\beta+1} - \sum_{n=1}^\infty \frac{\tilde\varepsilon_n(x)}{(-\beta)^n}$, and $\tilde\varepsilon_n(x) = \varepsilon_n(\phi(x))$.
Note that $\widetilde{T}_{-\beta}(x) = -\beta x - \lfloor -\beta x \rfloor$ except for finitely many points, hence $\widetilde{T}_{-\beta}$ is a natural generalization of the beta-transformation.
The map $\widetilde{T}_{-\beta}$ was studied e.g.\ by G\'ora~\cite{Gora07}, where it corresponds to the case $E = [1,1,\ldots,1]$, and in \cite{Liao-Steiner}.
The following corollary is an immediate consequence of Theorems~\ref{t:gora} and~\ref{t:char}.

\begin{corollary}
Let $a_1 a_2 \cdots$ be a sequence of non-negative integers satisfying \eqref{e:gora} and~\eqref{e:LS}. 
Then there exists a unique $\beta > 1$ such that
\begin{equation} \label{e:closedinterval2}
- \sum_{j=1}^\infty \frac{a_j+1}{(-\beta)^j} = 1 \quad \mbox{and} \quad - \sum_{j=1}^\infty \frac{a_{k+j}+1}{(-\beta)^j} \in [0,1] \quad \mbox{for all}\ k \ge 1. 
\end{equation}
Moreover, $\sum_{j=1}^\infty \frac{a_{k+j}+1}{(-\beta)^j} \ne 0$ for all $k \ge 1$ if and only if \eqref{e:ak1} and~\eqref{e:ak2} hold.
\end{corollary} 

With the notation of~\cite{Gora07}, this means, for $E = [1,1,\ldots,1]$, that $a_1 a_2 \cdots$ is the itinerary $\mathrm{It}_\beta(1)$ for some $\beta > 1$ if and only if \eqref{e:gora}, \eqref{e:LS}, \eqref{e:ak1}, and~\eqref{e:ak2} hold.
Note that G\'ora \cite[Theorems~25 and~28]{Gora07} claims that already \eqref{e:gora} is sufficient when $a_1 \ge 2$, and he has a less explicit statement for $a_1 = 1$. 
However, his proof deals only with the first part of the theorem, i.e., that there exists a unique $\beta > 1$ satisfying~\eqref{e:closedinterval2}. 
To see that this is not sufficient, consider the sequences $a_1 a_2 \cdots \in \{2, 1\,0\}^\omega$.
They all satisfy \eqref{e:closedinterval2} with $\beta = 2$, and there are uncountably many of them satisfying~\eqref{e:gora} and $a_1 = 2$.
All these uncountably many sequences would have to be equal to $\mathrm{It}_2(1)$ by \cite[Theorem~25]{Gora07}, which is of course not true.
(See also~\cite{Dombek-Masakova-Pelantova11}.)
Moreover, G\'ora's proof of the existence of a unique $\beta > 1$ satisfying~\eqref{e:closedinterval2} is incorrect when $\beta$ is small, see Remark~\ref{r:order}.

\section{Proof of Theorem~\ref{t:order}}
Let $\beta > 1$. 
For a sequence of digits $b_1 \cdots b_n$, set 
\[
I_{b_1 \cdots b_n} = \big\{x \in \big[\tfrac{-\beta}{\beta+1}, \tfrac{1}{\beta+1}\big):\, \varepsilon_1(x) \cdots \varepsilon_n(x) = b_1 \cdots b_n\big\}\,,
\]
with $\varepsilon_j(x)$ as in~\eqref{e:epsilon}.
Let $L_{\beta,n}$ be the number of different sequences $b_1 \cdots b_n$ such that $I_{b_1 \cdots b_n} \ne \emptyset$, and let $L'_{\beta,n}$ be the number of different sequences $b_1 \cdots b_n$ such that $I_{b_1 \cdots b_n}$ is an interval of positive length. 
(The latter is called the lap number of~$T_{-\beta}^n$.)

\begin{lemma} \label{l:entropy}
For any $\beta > 1$, we have that $\lim_{n\to\infty} \frac{1}{n} \log L_{\beta,n} = \lim_{n\to\infty} \frac{1}{n} \log L'_{\beta,n} = \log \beta$.
\end{lemma}

\begin{proof}
It is well known that the entropy of~$T_{-\beta}$, which is a piecewise linear map of constant slope~$-\beta$, is $\log \beta$.
The lemma can be derived from this fact, see~\cite{Frougny-Lai11}, but we prefer giving a short elementary proof, following Faller \cite[Proposition~3.6]{Faller08}. 
As $\big|\frac{d}{dx} T_{-\beta}^n(x)\big| = \beta^n$ at all points of continuity of~$T_{-\beta}^n$, the length of any interval $I_{b_1 \cdots b_n}$ is at most~$\beta^{-n}$.
Since the intervals $I_{b_1 \cdots b_n}$ form a partition of an interval of length~$1$, we obtain that $L_{\beta,n} \ge L_{\beta,n}' \ge \beta^n$. 

To get an upper bound for~$L'_{\beta,n}$, let $m$ be the smallest positive integer such that $\beta^m > 2$, and let $\delta$ be the minimal positive length of an interval $I_{b_1 \cdots b_m}$. 
Consider an interval~$I_{b_1 \cdots b_n}$, $n > m$, such that $b_1 \cdots b_n$ is neither the minimal nor the maximal sequence (with respect to the alternating lexicographic order) starting with $b_1 \cdots b_{n-m}$ and satisfying $I_{b_1 \cdots b_n} \ne \emptyset$.
Then each prolongation $b_1 b_2 \cdots$ satisfies the inequalities in \eqref{e:IS1} and \eqref{e:IS2}, respectively, for $1 \le k \le n-m$. 
Therefore, $b_1 b_2 \cdots$ is $(-\beta)$-admissible if and only if $b_{n-m+1} b_{n-m+2} \cdots$ is $(-\beta)$-admissible.
This implies that $T_{-\beta}^{n-m}(I_{b_1 \cdots b_n}) = I_{b_{n-m+1} \cdots b_m}$, and the length of $I_{b_1 \cdots b_n}$ is $\beta^{m-n}$ times the length of~$I_{b_{n-m+1} \cdots b_m}$, thus at least $\beta^{m-n} \delta$ when the length is positive.
There are at least $L'_{\beta,n} - 2 L'_{\beta,n-m}$ sequences $b_1 \cdots b_n$ such that $I_{b_1 \cdots b_n}$ has positive length and $b_1 \cdots b_n$ is neither the minimal nor the maximal sequence starting with $b_1 \cdots b_{n-m}$ and satisfying $I_{b_1 \cdots b_n} \ne \emptyset$.
This yields that $(L'_{\beta,n} - 2 L'_{\beta,n-m}) \beta^{m-n} \delta \le 1$ for all $n > m$, thus
\begin{align*}
L'_{\beta,n} & \le \frac{\beta^{n-m}}{\delta} + 2 L'_{\beta,n-m} \le \frac{\beta^{n-m}}{\delta} + \frac{2\beta^{n-2m}}{\delta} + 4 L'_{\beta,n-2m} \le \cdots \\
& \hspace{-1.2em} \le \frac{\beta^{n-m}}{\delta} \sum_{j=0}^{\lceil n/m\rceil-2} \bigg(\frac{2}{\beta^m}\bigg)^j + 2^{\lceil n/m\rceil-1} L'_{\beta,n-\lceil n/m\rceil m+m} < \frac{\beta^n}{\delta}\, \frac{1}{\beta^m-2} + \beta^n L'_{\beta,m} \le \frac{\beta^n}{\delta}\, \frac{\beta^m-1}{\beta^m-2}\,.
\end{align*}
This shows that $\lim_{n\to\infty} \frac{1}{n} \log L'_{\beta,n} = \beta$.

An interval $I_{b_1 \cdots b_n}$ consists only of one point if and only if $I_{b_1 \cdots b_k} = \big\{\frac{-\beta}{\beta+1}\big\}$ and $b_{k+1} \cdots b_n = a_1 \cdots a_{n-k}$ for some $k \le n$.
(This can happen only in case that $a_1 a_2 \cdots$ is periodic with odd period length.) 
Therefore, we can estimate $L_{\beta,n} - L'_{\beta,n} \le L'_{\beta,0} + L'_{\beta,1} + \cdots + L'_{\beta,n} \le C \beta^n$ for some constant $C > 0$, thus $\lim_{n\to\infty} \frac{1}{n} \log L_{\beta,n} = \lim_{n\to\infty} \frac{1}{n} \log L'_{\beta,n}$.
\end{proof}

For the proof of Theorem~\ref{t:order}, let $a_1 a_2 \cdots$ be the $(-\beta)$-expansion of~$\frac{-\beta}{\beta+1}$ and $a'_1 a'_2 \cdots$ be the $(-\beta')$-expansion of~$\frac{-\beta'}{\beta'+1}$, $\beta, \beta' > 1$.
If $\beta = \beta'$, then we clearly have that $a_1 a_2 \cdots = a'_1 a'_2 \cdots$.
If $a_1 a_2 \cdots = a'_1 a'_2 \cdots$, then the $(-\beta)$-admissible sequences are equal to the $(-\beta')$-admissible sequences, thus $L_{\beta,n} = L_{\beta',n}$ for all $n \ge 1$, and $\beta = \beta'$ by Lemma~\ref{l:entropy}.
Therefore, the equations $\beta = \beta'$ and $a_1 a_2 \cdots = a'_1 a'_2 \cdots$ are equivalent.
Hence, it suffices to show that $a_1 a_2 \cdots <_\mathrm{alt} a'_1 a'_2 \cdots$ implies that $\beta < \beta'$, as the other direction follows by contraposition.

Assume that $a_1 a_2 \cdots <_\mathrm{alt} a'_1 a'_2 \cdots$, and let $b_1 b_2 \cdots$ be a $(-\beta)$-admissible sequence. 
By \eqref{e:IS1} and \eqref{e:IS2} respectively, we have that
\begin{equation} \label{e:smaller}
b_k b_{k+1} \cdots \le_\mathrm{alt} a_1 a_2 \cdots <_\mathrm{alt} a'_1 a'_2 \cdots.
\end{equation}
Furthermore, as $\overline{0 a_1 \cdots a_{2\ell} ({a_{2\ell+1}\!-\!1})} >_\mathrm{alt} 0 a_1 a_2 \cdots$ for all $\ell \ge 0$, we obtain that
\begin{equation} \label{e:larger}
b_k b_{k+1} \cdots >_\mathrm{alt} 0 a_1 a_2 \cdots >_\mathrm{alt} 0 a'_1 a'_2 \cdots.
\end{equation}
If $a'_1 a'_2 \cdots$ is not periodic with odd period length, then \eqref{e:smaller} and \eqref{e:larger} show that $b_1 b_2 \cdots$ is $(-\beta')$-admissible, thus $L_{\beta,n} \le L_{\beta',n}$ for all $n \ge 1$, and $\beta \le \beta'$ by Lemma~\ref{l:entropy}.
Since $a_1 a_2 \cdots \ne a'_1 a'_2 \cdots$, this yields that $\beta < \beta'$. 
In case $a'_1 a'_2 \cdots = \overline{a'_1 \cdots a'_{2\ell'+1}}$, we show that 
\begin{equation} \label{e:larger2}
a_1 a_2 \cdots \le_\mathrm{alt} \overline{a'_1 \cdots a'_{2\ell'} ({a'_{2\ell'+1}\!-\!1}) 0}.
\end{equation}
This is clearly true when $a_1 \cdots a_{2\ell'+1} <_\mathrm{alt} a'_1 \cdots a'_{2\ell'} ({a'_{2\ell'+1}\!-\!1})$.
If $a_1 \cdots a_{2\ell'+1} = a'_1 \cdots a'_{2\ell'+1}$, then $a_{2\ell'+2} a_{2\ell'+3} \cdots >_\mathrm{alt} a'_{2\ell'+2} a'_{2\ell'+3} \cdots = a'_1 a'_2 \cdots >_\mathrm{alt} a_1 a_2 \cdots$, contradicting~\eqref{e:gora}.
It remains to consider the case that $a_1 \cdots a_{2\ell'+1} = a'_1 \cdots a'_{2\ell'} ({a'_{2\ell'+1}\!-\!1})$.
If $a_{2\ell'+1} > 0$, then \eqref{e:larger2} holds, otherwise $a_1 \cdots a_{2\ell'+2} = a'_1 \cdots a'_{2\ell'} ({a'_{2\ell'+1}\!-\!1}) 0$.
In the latter case, \eqref{e:gora}~implies that $a_{2\ell'+3} \cdots a_{4\ell'+4} \le_\mathrm{alt} a_1 \cdots a_{2\ell'+2} = a'_1 \cdots a'_{2\ell'} ({a'_{2\ell'+1}\!-\!1}) 0$, and we obtain inductively that \eqref{e:larger2} holds.
Now, \eqref{e:smaller}, \eqref{e:larger}, and~\eqref{e:larger2} show that $b_1 b_2 \cdots$ is $(-\beta')$-admissible, which yields as above that $\beta < \beta'$.

\section{Proof of Theorem~\ref{t:gora}} \label{sec:proof-theor-reft:g}

Let $a_1 a_2 \cdots$ be a sequence of non-negative integers satisfying \eqref{e:gora} and~\eqref{e:LS}. 
We show that there exists a unique $\beta > 1$ satisfying~\eqref{e:closedinterval2}, which is equivalent to~\eqref{e:closedinterval}.
For $n \ge 1$,~set
\begin{gather*}
P_n(x) = (-x)^n + \sum_{j=1}^n (a_j+1)\, (-x)^{n-j}, \\
J_n = \big\{x > 1 \mid P_j(x) \in [0,1]\ \mbox{for all}\ 1 \le j \le n\big\}.
\end{gather*}
Then $J_1 \supseteq J_2 \supseteq J_3 \supseteq \cdots$, and $J_n$ is compact if and only if $\inf J_n \ne 1$. 

First note that, for $\beta > 1$, \eqref{e:closedinterval2} is equivalent to $\beta \in \bigcap_{n\ge1} J_n$.
Indeed, if \eqref{e:closedinterval2} holds, then $P_n(\beta) = - \sum_{j=1}^\infty \frac{a_{n+j}+1}{(-\beta)^j} \in [0,1]$ for all $n \ge 1$. 
On the other hand, if $P_n(\beta) \in [0,1]$ for all $n \ge 1$, then $\big|1 + \sum_{j=1}^\infty \frac{a_j+1}{(-\beta)^j}\big| = \lim_{n\to\infty} \frac{P_n(\beta)}{(-\beta)^n} = 0$, thus \eqref{e:closedinterval2} holds.

Inductively for $n \ge 1$, we show the following statements, where we use the abbreviations $v_{[j,k]}$ for $v_j v_{j+1} \cdots v_k$ and $v_{[j,k)}$~for $v_j v_{j+1} \cdots v_{k-1}$:
\begin{enumerate}
\item \label{1}
$J_n$ is a non-empty interval, with $\inf J_n = 1$ if and only if $a_{[1,n]} = u_{[1,n]}$. \newline
If $P_n(\beta) = P_n(\beta') \in \{0,1\}$ with $\beta, \beta' \in J_n$, then $\beta = \beta'$. 
\item \label{2}
If $n$ is even, $a_{[1,n-2m+1]} = u_{[1,n-2m+1]}$ or $a_{[n-2m+2,n]} \ne a_{[1,2m)}$ for all $1 \le m \le n/2$, and $a_{[1,n]} \ne u_{[1,n]}$, then $P_n(\min J_n) = 0$. \newline
If $n$ is odd and $a_{[n-2m+2,n]} \ne a_{[1,2m)}$ for all $1 \le m \le n/2$, then $P_n(\max J_n) = 0$.
\item \label{3}
If $n$ is even, $a_{[1,n-2m+1]} \ne u_{[1,n-2m+1]}$ and $a_{[n-2m+2,n]} = a_{[1,2m)}$ for some $1 \le m \le n/2$, and $m$ is maximal with this property, then $P_n(\min J_n) = P_{2m-1}(\min J_n)$. \newline
If $n$ is odd, $a_{[n-2m+2,n]} = a_{[1,2m)}$ for some $1 \le m \le n/2$, and $m$ is maximal with this property, then $P_n(\max J_n) = P_{2m-1}(\max J_n)$.
\item \label{4}
If $n$ is even and $a_{[n-2m+1,n]} \ne a_{[1,2m]}$ for all $1 \le m < n/2$, then $P_n(\max J_n) = 1$. \newline
If $n$ is odd, $a_{[1,n-2m]} = u_{[1,n-2m]}$ or $a_{[n-2m+1,n]} \ne a_{[1,2m]}$ for all $1 \le m < n/2$, and $a_{[1,n]} \ne u_{[1,n]}$, then $P_n(\min J_n) = 1$.
\item \label{5}
If $n$ is even, $a_{[n-2m+1,n]} = a_{[1,2m]}$ for some $1 \le m < n/2$, and $m$ is maximal with this property, then $P_n(\max J_n) = P_{2m}(\max J_n)$. \newline
If $n$ is odd, $a_{[1,n-2m]} \ne u_{[1,n-2m]}$ and $a_{[n-2m+1,n]} = a_{[1,2m]}$ for some $1 \le m < n/2$, and $m$ is maximal with this property, then $P_n(\min J_n) = P_{2m}(\min J_n)$.
\end{enumerate}

We have that $P_1(x) = a_1+1 - x$, and $a_1 \ge 1$ by~\eqref{e:LS}. 
If $a_1 \ge 2$, then $J_1 = [a_1, a_1+1]$, $P_1(a_1) = 1$ and $P_1(a_1+1) = 0$; if $a_1 = 1$, then $J_1 = (1,2]$ and $P_1(2) = 0$.
Therefore, the statements hold for $n = 1$. 
Assume that they hold for $n-1$, and set 
\[
B = \big\{b \in \{0,1,\ldots,a_1\}:\, b+1 - x P_{n-1}(x) \in [0,1]\ \mbox{for some}\ x \in J_{n-1}\big\},
\]
i.e., $J_n \ne \emptyset$ if and only if $a_n \in B$.

Assume first that $a_{[1,n)} \ne u_{[1,n)}$, i.e., $\inf J_{n-1} = \min J_{n-1} > 1$, and that $n$ is even.
\renewcommand{\theenumi}{\roman{enumi}}
\begin{enumerate}
\item \label{i}
If $a_{[n-2m+1,n)} \ne a_{[1,2m)}$ for all $1 \le m < n/2$, then $P_{n-1}(\max J_{n-1}) = 0$, thus 
\[
1 - (\max J_{n-1})\, P_{n-1}(\max J_{n-1}) = 1.
\]
This implies that $0 \in B$, and $P_n(\max J_n) = P_n(\max J_{n-1}) = 1$ if $a_n = 0$. 
Since the map $x \mapsto x P_{n-1}(x)$ is continuous and $J_{n-1}$ is an interval, we get that \mbox{$P_n(\max J_n) = 1$} for $a_n > 0$ as well, when $J_n \ne \emptyset$.
Moreover, we clearly have that $a_{[n-2m+1,n]} \ne a_{[1,2m]}$ for all $1 \le m < n/2$, thus (\ref{4}) holds when $a_n \in B$. 
\item \label{ii}
If $a_{[n-2m+1,n)} = a_{[1,2m)}$ for some $1 \le m < n/2$, and $m$ is maximal with this property, then $P_{n-1}(\max J_{n-1}) = P_{2m-1}(\max J_{n-1})$, thus
\[
a_{2m}+1 - (\max J_{n-1})\, P_{n-1}(\max J_{n-1}) = P_{2m}(\max J_{n-1}) \in [0,1],
\]
where we have used that $J_{n-1} \subseteq J_{2m}$ and $P_{2m}(J_{2m}) \subseteq [0,1]$.
This gives $a_{2m} \in B$. \newline
If $a_n = a_{2m}$, then $\max J_n = \max J_{n-1}$ and $P_n(\max J_{n-1}) = P_{2m}(\max J_{n-1})$, thus $P_n(\max J_n) = P_{2m}(\max J_n)$ and $a_{[n-2m+1,n]} = a_{[1,2m]}$.
By the maximality of~$m$, we have that $a_{[n-2\ell+1,n]} \ne a_{[1,2\ell]}$ for all $m < \ell < n/2$, thus (\ref{5}) holds. \newline
If $a_n \ne a_{2m}$, then the equation $a_{[n-2m+1,n)} = a_{[1,2m)}$ and \eqref{e:gora} yield that $a_n > a_{2m}$, thus $P_n(\max J_n) = 1$ when $J_n \ne \emptyset$, similarly to~(\ref{i}). 
If $a_{[1,2\ell)} = a_{[n-2\ell+1,n)}$, $1 \le \ell < m$, then we also have that $a_{[1,2\ell)} = a_{[2m-2\ell+1,2m)}$, thus $a_{2\ell} \le a_{2m} < a_n$.
This implies that $a_{[n-2\ell+1,n]} \ne a_{[1,2\ell]}$ for all $1 \le \ell < n/2$, thus (\ref{4}) holds when $a_n \in B$.
\item \label{iii}
If $a_{[1,n-2m)} = u_{[1,n-2m)}$ or $a_{[n-2m,n)} \ne a_{[1,2m]}$ for all $1 \le m \le n/2 - 1$, then we have that $P_{n-1}(\min J_{n-1}) = 1$, thus
\[
a_1+1 - (\min J_{n-1})\, P_{n-1}(\min J_{n-1}) = P_1(\min J_{n-1}) \in [0,1],
\]
and $a_1 \in B$.
If $a_n = a_1$, then $\min J_n = \min J_{n-1}$ and $P_n(\min J_{n-1}) = P_1(\min J_{n-1})$, thus $P_n(\min J_n) = P_1(\min J_n)$, and $a_{[1,n-2m+1]} = u_{[1,n-2m+1]}$ or $a_{[n-2m+2,n]} \ne a_{[1,2m)}$ for all $2 \le m \le n/2$.
Therefore, (\ref{3})~holds.
If $a_n < a_1$, then $P_n(\min J_n) = 0$ when $J_n \ne \emptyset$, $a_{[1,n-2m+1]} = u_{[1,n-2m+1]}$ or $a_{[n-2m+2,n]} \ne a_{[1,2m)}$ for all $1 \le m \le n/2$, thus (\ref{2}) holds when $a_n \in B$.
\item \label{iv}
If $a_{[1,n-2m)} \ne u_{[1,n-2m)}$ and $a_{[n-2m,n)} = a_{[1,2m]}$ for some $1 \le m \le n/2-1$, and $m$ is maximal with this property, then $P_{n-1}(\min J_{n-1}) = P_{2m}(\min J_{n-1})$, thus
\[
a_{2m+1}+1 - (\min J_{n-1})\, P_{n-1}(\min J_{n-1}) = P_{2m+1}(\min J_{n-1}) \in [0,1],
\]
hence $a_{2m+1} \in B$. 
If $a_n = a_{2m+1}$, then $\min J_n = \min J_{n-1}$ and $P_n(\min J_{n-1}) = P_{2m+1}(\min J_{n-1})$, thus $P_n(\min J_n) = P_{2m+1}(\min J_n)$, and $a_{[n-2m,n]} = a_{[1,2m+1]}$.
The maximality of~$m$ yields that $a_{[1,n-2\ell+1]} = u_{[1,n-2\ell+1]}$ or $a_{[n-2\ell+2,n]} \ne a_{[1,2\ell)}$ for all $m+1 < \ell \le n/2$, thus (\ref{3})~holds. 
If $a_n \ne a_{2m+1}$, then $a_n < a_{2m+1}$ by~\eqref{e:gora}.
If moreover $a_{[1,2\ell-2]} = a_{[n-2\ell+2,n)}$, $1 \le \ell \le m$, then we have that $a_{[1,2\ell-2]} = a_{[2m-2\ell+3,2m]}$, thus $a_{2\ell-1} \ge a_{2m+1} > a_n$.
Then we get that $P_n(\min J_n) = 0$ when $J_n \ne \emptyset$, $a_{[1,n-2\ell+1]} = u_{[1,n-2\ell+1]}$ and $a_{[n-2\ell+2,n]} \ne a_{[1,2\ell)}$ for all $1 \le \ell \le n/2$, thus (\ref{2}) holds when $a_n \in B$. 
\end{enumerate}
Since $x \mapsto x P_{n-1}(x)$ is continuous and $J_{n-1}$ is an interval, the set $B$ is an interval of integers.
The paragraphs (\ref{i}) and (\ref{ii}) show that $a_n$ is not smaller than the smallest element of~$B$, (\ref{iii}) and (\ref{iv}) show that $a_n$ is not larger than the largest element of~$B$, thus $a_n \in B$.
We have therefore proved that $J_n \ne \emptyset$ and (\ref{2})--(\ref{5}) hold, when $a_{[1,n)} \ne u_{[1,n)}$ and $n$ is even.
For odd~$n$, the proof runs along the same lines and is left to the reader.

If $a_{[1,n)} = u_{[1,n)}$, then $\inf J_{n-1} = 1$.
From \cite[Proposition~3.5]{Liao-Steiner}, we know that $u_n \in B$, that $\inf J_n = 1$ when $a_n = u_n$, and that $\min J_n > 1$ when $u_n \ne a_n \in B$.
Let first $n$ be even, thus $a_n \le u_n$ by~\eqref{e:LS}.
If $a_{[n-2m+1,n)} \ne a_{[1,2m)}$ for all $1 \le m < n/2$, then we obtain as in~(\ref{i}) that $0 \in B$, thus $a_n \in B$, and (\ref{4})~holds.
If $a_{[n-2m+1,n)} = a_{[1,2m)}$ for some $1 \le m < n/2$, and $m$ is maximal with this property, then (\ref{ii}) yields that $a_{2m} \in B$ and $a_{2m} \le a_n$, thus $a_n \in B$. 
If $a_n = a_{2m}$, then (\ref{5}) holds; if $a_n > a_{2m}$, then (\ref{4}) holds.
Moreover, if $a_n < u_n$, then we get that $P_n(\min J_n) = 0$, thus (\ref{2}) holds. 
Again, if $n$ is odd, then similar arguments apply. 
Hence, we have proved that $J_n \ne \emptyset$ and (\ref{2})--(\ref{5}) hold for the case that $a_{[1,n)} = u_{[1,n)}$ too.

If $J_n$ is not an interval, then the continuity of $x \mapsto x P_{n-1}(x)$ on the interval~$J_{n-1}$ implies that $P_n$ meets the lower bound~$0$ or the upper bound~$1$ at least twice within~$J_n$. 
Therefore, suppose that $P_n(\beta) = P_n(\beta') \in \{0,1\}$ for $\beta, \beta' \in J_n$.
If $P_j(\beta) \in (0,1]$ and $P_j(\beta') \in (0,1]$ for all $1 \le j < n$, then the $(-\beta)$-expansion of~$\frac{-\beta}{\beta+1}$ and the $(-\beta')$-expansion of~$\frac{-\beta'}{\beta'+1}$ are both $\overline{a_{[1,n]}}$ (if $P_n(\beta) = 1$) or $\overline{a_{[1,n)} ({a_n\!+\!1})}$  (if $P_n(\beta) = 0$), thus $\beta = \beta'$ by Theorem~\ref{t:order}.

Suppose in the following that $P_j(\beta') = 0$ for some $1 \le j < n$, and let $\ell \ge 1$ be minimal such that $P_\ell(\beta') \in \{0,1\}$.
If $P_\ell(\beta') = 0$, then $a_{\ell+1} = 0$ and $P_{\ell+1}(\beta') = 1$, hence $a_{[1,n]}$ is a concatenation of blocks $a_{[1,\ell]} 0$ and $a_{[1,\ell)} ({a_\ell\!+\!1})$, except possibly for the last block, which is $a_{[1,\ell]}$ when $P_n(\beta') = 0$.
If $P_\ell(\beta') = 1$, then $a_{[1,n]}$ is a concatenation of blocks $a_{[1,\ell]}$ and $a_{[1,\ell)} ({a_\ell\!-\!1}) 0$, ending with $a_{[1,\ell)} ({a_\ell\!-\!1})$ when $P_n(\beta') = 0$.
We obtain that
\[
P_n(x) = P_n(\beta') + \big(P_\ell(x) - P_\ell(\beta')\big)\, Q(x)
\]
for some polynomial $Q(x) = \sum_{j=0}^{n-\ell} q_j\, (-x)^j$ with coefficients $q_j \in \{0,1\}$, and $q_{j-1} = q_{j-2} = \cdots = q_{j-\ell+1} = 0$ whenever $q_j = 1$.
If $P_\ell(\beta) = P_\ell(\beta')$, then the induction hypotheses yield that $\beta = \beta'$.
If $P_\ell(\beta) \ne P_\ell(\beta')$, then $Q(\beta) = 0$, which implies that $1 < \frac{1}{\beta^{\ell+1}} + \frac{1}{\beta^{2\ell+1}} + \cdots = \frac{1}{\beta^{\ell+1}-\beta}$ when $\ell$ is even, $1 < \frac{1}{\beta^\ell} + \frac{1}{\beta^{2\ell+1}} + \cdots = \frac{\beta}{\beta^{\ell+1}-1}$ when $\ell$ is odd, i.e., $\beta^{\ell+1} < \beta + 1$.

To exclude the latter case, suppose that $P_n(\beta) = P_n(\beta') \in \{0,1\}$ for $\beta, \beta' \in J_n$, $\beta \ne \beta'$, and that $\beta^{\ell+1} < \beta + 1$ for the minimal $\ell \ge 1$ such that $P_\ell(\beta') \in \{0,1\}$.
Set $g_k = \lfloor 2^{k+1}/3\rfloor$, and let, for $k \ge 1$, $\gamma_k$~and $\eta_k$ be the real numbers greater than~$1$ satisfying $\gamma_k^{g_k+1} = \gamma_k + 1$, $\eta_k^{g_k+1} = \eta_k^{g_{k-1}+1} + 1$ when $k$ is even, $\eta_k^{g_k} = \eta_k^{g_{k-1}} + 1$ when $k$ is odd, as in~\cite{Liao-Steiner}.
For the positive integer~$m$ satisfying $g_m \le \ell < g_{m+1}$, we have that $\beta < \gamma_m < \eta_m$.
By Proposition~3.5 in~\cite{Liao-Steiner} and its proof, $\beta < \eta_m$ implies that the $(-\beta)$-expansion of~$\frac{-\beta}{\beta+1}$ starts with~$\varphi^m(1)$ and that $\widetilde{T}_{-\beta}^j(1) \not\in \{0,1\}$ for all $1 \le j \le |\varphi^m(1)| = g_{m+1} + \frac{1-(-1)^m}{2}$, where $|w|$ denotes the length of the word~$w$.
Since $\beta \in J_n$ and $P_n(\beta) \in \{0,1\}$, we obtain that $a_1 a_2 \cdots$ starts with~$\varphi^m(1)$ and that $n > |\varphi^m(1)|$.
By equation~(3.2) in \cite{Liao-Steiner}, we have that $P_{2^m}(x) > 1$ for all $x > \eta_m$ (note that $2^m = |\varphi^{m-1}(10)| < |\varphi^m(1)|$), thus $J_{2^m} = (1, \eta_m]$, and $\ell < g_{m+1}$ yields that $\beta' = \eta_m$, $\ell = 2^m$.
As~$\beta$ and $\beta'$ are in the interval~$J_{n-1}$, we also have that $\gamma_m \in J_{n-1}$. 
The $(-\gamma_m)$-expansion of~$\frac{-\gamma_m}{\gamma_m+1}$ is $\varphi^{m-1}(1)\, \overline{\varphi^{m-1}(0)}$ by \cite[Theorem~2.5]{Liao-Steiner}.
Since $n \ge 2 \ell$ by the above block decomposition of~$a_{[1,n]}$, we obtain that $a_1 a_2 \cdots$ starts with~$\varphi^{m-1}(1000)$ if $m \ge 2$, and with $100$ if $m = 1$. 
In case $m = 1$, we get that $P_3(2) \not\in J_3$, contradicting that $2 = \eta_1 = \beta' \in J_n$. 
For $m \ge 2$, we have that $P_{|\varphi^{m-1}(1000)|}(\eta_m)  > P_{|\varphi^{m-1}(10)|}(\eta_m) = 1$ because $P_{|\varphi^{m-1}(100)|}(\eta_m) = P_{|\varphi^m(1)|}(\eta_m) < P_{|\varphi^{m-1}(1)|}(\eta_m)$ by equation~(3.4) in~\cite{Liao-Steiner} and, using the notation of~\cite{Liao-Steiner}, the function $f_{\gamma_m,\varphi^{m-1}(0)}$ is order-reversing. 
Again, this contradicts that $\eta_m = \beta' \in J_n$.
Therefore, we have shown that $\beta = \beta'$ whenever $P_n(\beta) = P_n(\beta') \in \{0,1\}$, $\beta, \beta' \in J_n$.
Hence, $J_n$ is an interval, and (\ref{1})--(\ref{5}) hold for all $n \ge 1$.

As the $J_n$ form a sequence of nested non-empty intervals that are compact for sufficiently large~$n$, we have that $\bigcap_{n\ge1} J_n \ne \emptyset$, thus there exists some $\beta > 1$ satisfying~\eqref{e:closedinterval2}, which is equivalent to~\eqref{e:closedinterval}.
To show that $\beta$ is unique, suppose that $\bigcap_{n\ge1} J_n$ is not a single point. 
Then $\bigcap_{n\ge1} J_n$ is an interval of positive length, thus there exist $\beta, \beta' \in \bigcap_{n\ge1} J_n$, $\beta \ne \beta'$, such that $P_n(\beta) \in (0,1]$ and $P_n(\beta') \in (0,1]$ for all $n \ge 1$.
This means that $a_1 a_2 \cdots$ is both the $(-\beta)$-expansion of~$\frac{-\beta}{\beta+1}$ and the $(-\beta')$-expansion of~$\frac{-\beta'}{\beta'+1}$, which contradicts that $\beta \ne \beta'$ by Theorem~\ref{t:order}.
This concludes the proof of Theorem~\ref{t:gora}.

\begin{remark} \label{r:order}
Some parts of the proofs of Theorems~\ref{t:gora} and~\ref{t:order} can be simplified when one is only interested in $\beta > 1$ not too close to~$1$.
Since $P_n(x) = a_n + 1 - x P_{n-1}(x)$ for $n \ge 2$, and $P'_1(x) = -1$, the derivative of $P_n(x)$ is
\[
P'_n(x) = (-1)\, \big(P_{n-1}(x) + x P'_{n-1}(x)\big) = \cdots = (-1)^n x^{n-1} \bigg(1 + \sum_{j=1}^{n-1} \frac{P_j(x)}{(-x)^j}\bigg).
\]
If $x \in J_{n-1}$, then $1 + \sum_{j=1}^{n-1} \frac{P_j(x)}{(-x)^j} > 1 - \frac{1}{x} - \frac{1}{x^3} - \cdots = \frac{x^2-x-1}{x^2-1}$.
If moreover $x \ge ({1\!+\!\sqrt{5}})/2$, then we get that $(-1)^n P'_n(x) > 0$, hence $P_n$ is a strictly increasing (decreasing) function on $J_{n-1} \cap [({1\!+\!\sqrt{5}})/2, \infty)$ when $n$ is even (odd).
Moreover, $\lim_{n\to\infty} |P'_n(x)| = \infty$ if $x \ge ({1\!+\!\sqrt{5}})/2$ and $x \in J_n$ for all $n \ge 1$.

However, it is not true that $P_n$ is always increasing (decreasing) on $J_{n-1}$ when $n$ is even (odd).
For instance, if $a_1 a_2 \cdots$ starts with $1 0 0 1$, then $P_4(x) = x^4 - 2x^3 + x^2 - x + 2$ and $J_3 = (1, \beta]$ with $\beta^3 = 2\beta^2 - \beta + 1$ ($\beta \approx 1.755$). 
The function $P_4$ decreases on $(1,\beta']$, with $\beta' \approx 1.261$, and increases on $[\beta', \infty)$. 
Note that this is a major flaw in the proof of Theorem~28 of \cite{Gora07} (besides the fact that the statement is incorrect, as explained in the Introduction).
This lack of monotonicity is what makes Theorems~\ref{t:gora} and~\ref{t:order} more difficult to prove than the corresponding statements for $\beta$-expansions.
\end{remark}

\section{Proof of Theorem~\ref{t:char}} \label{sec:proof-theor-reft:ch}
Let $a_1 a_2 \cdots$ be a sequence of non-negative integers satisfying \eqref{e:gora} and~\eqref{e:LS}.
We have already seen in the Introduction that these conditions are necessary to be the $(-\beta)$-expansion of~$\frac{-\beta}{\beta+1}$ for some $\beta > 1$.
Moreover, $\beta$ can only be the number given by Theorem~\ref{t:gora}.
Then $a_1 a_2 \cdots$ is the $(-\beta)$-expansion of $\frac{-\beta}{\beta+1}$ if and only if $\sum_{j=1}^\infty \frac{a_{k+j}}{(-\beta)^j} \ne \frac{1}{\beta+1}$ for all $k \ge 1$.  

Suppose first that $\sum_{j=1}^\infty \frac{a_{k+j}}{(-\beta)^j} = \frac{1}{\beta+1}$ for some $k \ge 1$, and let $\ell \ge 1$ be minimal such that $\sum_{j=1}^\infty \frac{a_{\ell+j}}{(-\beta)^j} \in \big\{\frac{-\beta}{\beta+1}, \frac{1}{\beta+1}\big\}$.
If $\sum_{j=1}^\infty \frac{a_{\ell+j}}{(-\beta)^j} = \frac{-\beta}{\beta+1}$, then the $(-\beta)$-expansion of $\frac{-\beta}{\beta+1}$ is $\overline{a_{[1,\ell]}}$.
Then $a_1 a_2 \cdots$ is composed of blocks $a_{[1,\ell]}$ and $a_{[1,\ell)} ({a_\ell\!-\!1}) 0$. 
Since $\sum_{j=1}^\infty \frac{a_{k+j}}{(-\beta)^j} = \frac{1}{\beta+1}$ for some $k \ge 1$, we have at least one block $a_{[1,\ell)} ({a_\ell\!-\!1}) 0$, i.e., $a_1 a_2 \cdots \in \{a_{[1,\ell]},\, a_{[1,\ell)} ({a_\ell\!-\!1}) 0\}^\omega \setminus \{\overline{a_{[1,\ell]}}\}$.
As $\overline{a_{[1,\ell]}}$ is the $(-\beta)$-expansion of $\frac{-\beta}{\beta+1}$, we have that $\overline{a_{[1,\ell]}} >_\mathrm{alt} u_1 u_2 \cdots$, thus \eqref{e:ak1} does not hold.
If $\sum_{j=1}^\infty \frac{a_{\ell+j}}{(-\beta)^j} = \frac{1}{\beta+1}$, then the $(-\beta)$-expansion of $\frac{-\beta}{\beta+1}$ is $\overline{a_{[1,\ell)} ({a_\ell\!+\!1})}$, $a_1 a_2 \cdots$ is composed of blocks $a_{[1,\ell]} 0$ and $a_{[1,\ell)} ({a_\ell\!+\!1})$, and we have that $\overline{a_{[1,\ell)} ({a_\ell\!+\!1})} >_\mathrm{alt} u_1 u_2 \cdots$, thus \eqref{e:ak2} does not hold.
Therefore, \eqref{e:gora}, \eqref{e:LS}, \eqref{e:ak1}, and \eqref{e:ak2} imply that $a_1 a_2 \cdots$ is the $(-\beta)$-expansion of~$\frac{-\beta}{\beta+1}$ for some (unique) $\beta > 1$.

Suppose now that \eqref{e:ak1} does not hold, i.e., $a_1 a_2 \cdots \in \{a_{[1,k]},\, a_{[1,k)} ({a_k\!-\!1}) 0\}^\omega \setminus \{\overline{a_{[1,k]}}\}$ for some $k \ge 1$ with $\overline{a_{[1,k]}} >_\mathrm{alt} u_1 u_2 \cdots$.
We show that the sequence $\overline{a_{[1,k]}}$ satisfies~\eqref{e:gora}.
Suppose on the contrary that $a_{[j,k]}\, \overline{a_{[1,k]}} >_\mathrm{alt} \overline{a_{[1,k]}}$ for some $2 \le j \le k$. 
This implies that $a_{[j,k]}\, a_{[1,j)} >_\mathrm{alt} a_{[1,k]}$.
Since $a_{[k+1,2k)} = a_{[1,k)}$, we obtain that $a_{[j,j+k)} = a_{[j,k]}\, a_{[1,j)} >_\mathrm{alt} a_{[1,k]}$, thus $a_j a_{j+1} \cdots >_\mathrm{alt} a_1 a_2 \cdots$, contradicting that $a_1 a_2 \cdots$ satisfies~\eqref{e:gora}.
Therefore, $\overline{a_{[1,k]}}$ satisfies~\eqref{e:gora} and~\eqref{e:LS}, and we can apply Theorem~\ref{t:gora} for this sequence.
Let $\beta' > 1$ be the number satisfying \eqref{e:closedinterval} for the sequence~$\overline{a_{[1,k]}}$.
Then $\beta'$ also satisfies \eqref{e:closedinterval} for the original sequence $a_1 a_2 \cdots$, thus $\beta' = \beta$.
Therefore, $a_1 a_2 \cdots$ is not the $(-\beta)$-expansion of~$\frac{-\beta}{\beta+1}$.  

Suppose finally that \eqref{e:ak2} does not hold, i.e., $a_1 a_2 \cdots \in \{a_{[1,k]} 0,\, a_{[1,k)} ({a_k\!+\!1})\}^\omega$ for some $k \ge 1$ with $\overline{a_{[1,k)} ({a_k\!+\!1})} >_\mathrm{alt} u_1 u_2 \cdots$.
If $a_1 a_2 \cdots = \overline{a_{[1,k]} 0}$, then $\sum_{j=1}^\infty \frac{a_{k+j}}{(-\beta)^j} = \frac{1}{\beta+1}$, thus $a_1 a_2 \cdots$ is not the $(-\beta)$-expansion of~$\frac{-\beta}{\beta+1}$.
If $a_1 a_2 \cdots \ne \overline{a_{[1,k]} 0}$, then we show that the sequence $\overline{a_{[1,k)} ({a_k\!+\!1})}$ satisfies~\eqref{e:gora}.
Suppose that $a_{[j,k)} ({a_k\!+\!1})\, \overline{a_{[1,k)} ({a_k\!+\!1})} >_\mathrm{alt} \overline{a_{[1,k)} ({a_k\!+\!1})}$ for some $2 \le j \le k$. 
This implies that $a_{[j,k)} ({a_k\!+\!1}) a_{[1,j)} >_\mathrm{alt} a_{[1,k]}$.
Since $a_{[j,k)} ({a_k\!+\!1}) a_{[1,j)} = a_{[\ell,\ell+k)}$ for some $\ell \ge 2$, we have that $a_\ell a_{\ell+1} \cdots >_\mathrm{alt} a_1 a_2 \cdots$, contradicting that $a_1 a_2 \cdots$ satisfies~\eqref{e:gora}.
As in the preceding paragraph, the number given by Theorem~\ref{t:gora} for the sequence $\overline{a_{[1,k)} ({a_k\!+\!1})}$ is~$\beta$, thus $a_1 a_2 \cdots$ is not the $(-\beta)$-expansion of~$\frac{-\beta}{\beta+1}$.  
Therefore, \eqref{e:ak1} and \eqref{e:ak2} are necessary for $a_1 a_2 \cdots$ to be the $(-\beta)$-expansion of~$\frac{-\beta}{\beta+1}$ for some $\beta > 1$.

\section*{Acknowledgments}
The author would like to thank Edita Pelantov\'a for many fruitful discussions.

\bibliographystyle{amsalpha}
\bibliography{yrrap}
\end{document}